\documentclass[10]{article}
\usepackage{amsfonts,amssymb,amsmath,amsthm,cite}
\usepackage{graphicx}

\usepackage[applemac]{inputenc}
\usepackage{amsmath,amssymb,amsthm, hyperref, euscript}
\usepackage[matrix,arrow,curve]{xy}
\usepackage{graphicx}
\usepackage{tabularx}
\usepackage{float}
\usepackage{hyperref}
\usepackage{tikz}
\usepackage{slashed}
\usepackage{mathrsfs}
\usepackage{multirow}

\usetikzlibrary{matrix}

\usepackage[T1]{fontenc}
\usepackage{amsfonts,cite}
\usepackage{graphicx}

\usepackage[applemac]{inputenc}

\usepackage[sc]{mathpazo}
\DeclareFontFamily{T1}{pzc}{}
\DeclareFontShape{T1}{pzc}{m}{it}{1.8 <-> pzcmi8t}{}
\DeclareMathAlphabet{\mathpzc}{T1}{pzc}{m}{it}
\theoremstyle{plain}
\newtheorem{prop}{Proposition}[section]

\newtheorem{lem}[prop]{Lemma}%[section]
\newtheorem{thm}[prop]{Theorem}%[section]
\newtheorem{theorem}[prop]{Theorem}

\newtheorem{corollary}[prop]{Corollary}

\theoremstyle{definition}
\newtheorem{defn}[prop]{Definition}%[section]
\theoremstyle{definition}
\newtheorem{definition}[prop]{Definition}

\newtheorem{remark}[prop]{Remark}

\numberwithin{equation}{section}

\DeclareMathOperator{\Dom}{Dom}              %% domain of an operator
\newcommand{\vertiii}[1]{{\left\vert\kern-0.25ex\left\vert\kern-0.25ex\left\vert #1
    \right\vert\kern-0.25ex\right\vert\kern-0.25ex\right\vert}}
\newcommand{\Coo}{C^\infty}                  %% smooth functions
\usepackage[sc]{mathpazo}
\newbox\ncintdbox \newbox\ncinttbox %% noncommutative integral symbols
\setbox0=\hbox{$-$} \setbox2=\hbox{$\displaystyle\int$}
\setbox\ncintdbox=\hbox{\rlap{\hbox
    to \wd2{\hskip-.125em \box2\relax\hfil}}\box0\kern.1em}
\setbox0=\hbox{$\vcenter{\hrule width 4pt}$}
\setbox2=\hbox{$\textstyle\int$} \setbox\ncinttbox=\hbox{\rlap{\hbox
    to \wd2{\hskip-.175em \box2\relax\hfil}}\box0\kern.1em}

\newcommand{\A}{\mathcal{A}}                 %% an algebra
\newcommand{\C}{\mathbb{C}}                  %% complex numbers
\renewcommand{\H}{\mathcal{H}}               %% Hilbert space
\newcommand{\hookto}{\hookrightarrow}        %% abbreviation
\newcommand{\N}{\mathbb{N}}                  %% nonnegative integers
\newcommand{\R}{\mathbb{R}}                  %% real numbers
\renewcommand{\SS}{\mathcal{S}}              %% Schwartz space
\newcommand{\Z}{\mathbb{Z}}                  %% integers
\newcommand{\al}{\alpha}          %% short for  \alpha
\newcommand{\bt}{\beta}           %% short for  \beta
\def\<#1|#2>{\langle#1\stroke#2\rangle} %% \braket (Dirac notation)
\def\?#1|#2?{\{#1\stroke#2\}}        %% left-linear pairing

\def\<#1,#2>{\langle#1,#2\rangle}            %% bilinear pairing
\def\ee_#1{e_{{\scriptscriptstyle#1}}}       %% basis projector
\def\wick:#1:{\mathopen:#1\mathclose:}       %% Wick-ordered operator

\newbox\ncintdbox \newbox\ncinttbox %% noncommutative integral symbols
\setbox0=\hbox{$-$}
\setbox2=\hbox{$\displaystyle\int$}
\setbox\ncintdbox=\hbox{\rlap{\hbox
           to \wd2{\box2\relax\hfil}}\box0\kern.1em}
\setbox0=\hbox{$\vcenter{\hrule width 4pt}$}
\setbox2=\hbox{$\textstyle\int$}
\setbox\ncinttbox=\hbox{\rlap{\hbox
           to \wd2{\hskip-.05em\box2\relax\hfil}}\box0\kern.1em}

\newcommand{\stroke}{\mathbin|}   %% (for `\bbraket' and such)
\newcommand{\SU}{SU}       %%
\newcommand{\be}{\begin{equation}}
\renewcommand{\ee}{\end{equation}}
\newcommand{\bea}{\begin{eqnarray}}
\newcommand{\eea}{\end{eqnarray}}
\newcommand{\bean}{\begin{eqnarray*}}
	\newcommand{\eean}{\end{eqnarray*}}
\newcommand{\brray}{\begin{array}}
	\newcommand{\erray}{\end{array}}

\title{Unoriented Spectral Triples}
\begin{document}
\maketitle  \setlength{\parindent}{0pt}
\begin{center}
\author{
{\textbf{Petr Ivankov*}\\
e-mail: * monster.ivankov@gmail.com \\
}
}
\end{center}

\vspace{1 in}

%\begin{abstract}
\noindent

\paragraph{}	
Any  oriented  Riemannian manifold with a Spin-structure defines a spectral triple, so the spectral triple can be regarded as a noncommutative Spin-manifold. Otherwise for any unoriented  Riemannian manifold there is the two-fold covering by oriented  Riemannian manifold. Moreover there are noncommutative generalizations of finite-fold coverings. This circumstances yield a notion of  unoriented spectral triple which is covered by oriented one. 

%\end{abstract}
%\tableofcontents

\section{Preliminaries}

\paragraph*{}
Gelfand-Na\u{\i}mark theorem \cite{arveson:c_alg_invt} states the correspondence between  locally compact Hausdorff topological spaces and commutative $C^*$-algebras. 
\begin{thm}\label{gelfand-naimark}\cite{arveson:c_alg_invt} (Gelfand-Na\u{i}mark). 
	Let $A$ be a commutative $C^*$-algebra and let $\mathcal{X}$ be the spectrum of A. There is the natural $*$-isomorphism $\gamma:A \to C_0(\mathcal{X})$.
\end{thm}
So any (noncommutative) $C^*$-algebra may be regarded as a generalized (noncommutative)  locally compact Hausdorff topological space. 
\subsection{Quantization of finite-fold coverings}
\paragraph*{}
Articles \cite{ivankov:qnc,pavlov_troisky:cov} contain noncommutative analogs of coverings of compact and noncompact spaces.
\begin{theorem}\label{pavlov_troisky_thm}\cite{pavlov_troisky:cov}
	Suppose $\mathcal X$ and $\mathcal Y$ are compact Hausdorff connected spaces and $p :\mathcal  Y \to \mathcal X$
	is a continuous surjection. If $C(\mathcal Y )$ is a projective finitely generated Hilbert module over
	$C(\mathcal X)$ with respect to the action
	\begin{equation*}
	(f\xi)(y) = f(y)\xi(p(y)), ~ f \in  C(\mathcal Y ), ~ \xi \in  C(\mathcal X),
	\end{equation*}
	then $p$ is a finite-fold  covering.
\end{theorem}
\begin{definition}
	If $A$ is a $C^*$- algebra then an action of a group $G$ is said to be {\it involutive } if $ga^* = \left(ga\right)^*$ for any $a \in A$ and $g\in G$. The action is said to be \textit{non-degenerated} if for any nontrivial $g \in G$ there is $a \in A$ such that $ga\neq a$. 
\end{definition}
\begin{definition}\label{fin_def_uni}
	Let $A \hookto \widetilde{A}$ be an injective *-homomorphism of unital $C^*$-algebras. Suppose that there is a non-degenerated involutive action $G \times \widetilde{A} \to \widetilde{A}$ of a finite group $G$, such that $A = \widetilde{A}^G\stackrel{\text{def}}{=}\left\{a\in \widetilde{A}~|~ a = g a;~ \forall g \in G\right\}$. There is an $A$-valued product on $\widetilde{A}$ given by
	\begin{equation}\label{finite_hilb_mod_prod_eqn}
	\left\langle a, b \right\rangle_{\widetilde{A}}=\sum_{g \in G} g\left( a^* b\right) 
	\end{equation}
	and $\widetilde{A}$ is an $A$-Hilbert module. We say that a triple $\left(A, \widetilde{A}, G \right)$ is an \textit{unital noncommutative finite-fold  covering} if $\widetilde{A}$ is a finitely generated projective $A$-Hilbert module.
\end{definition}
\begin{remark}
	Above definition is motivated by the Theorem \ref{pavlov_troisky_thm}.
\end{remark}
\begin{remark}\label{kasparov_stab_thm}
	From the Kasparov stabilization theorem \cite{blackadar:ko} it follows that any finitely generated $C^*$-Hilbert module is projective.  It follows that  "finitely generated projective" can be replaced with "finitely generated" in the Definition \ref{fin_def_uni}.
\end{remark}

\subsection{Spectral triples}

\paragraph{}
This section contains citations of  \cite{hajac:toknotes}. 
\begin{defn}
	\label{df:spec-triple}\cite{hajac:toknotes}
	An {\it unital (oriented) spectral triple} $(\A, \H, D)$ consists of:
	\begin{itemize}
		\item
		a pre-$C^*$-algebra $\A$ with an involution $a \mapsto a^*$, equipped
		with a faithful representation on:
		\item
		a \emph{Hilbert space} $\H$; and also
		\item
		a \emph{selfadjoint operator} $D$ on $\mathcal{H}$, with dense domain
		$\Dom D \subset \H$, such that $a(\Dom D) \subseteq \Dom D$ for all 
		$a \in \mathcal{A}$.
	\end{itemize}
\end{defn}
There is a set of axioms for  spectral triples described in \cite{hajac:toknotes,varilly:noncom}. 

\section{Main Definition}\label{main_defn_sec}
\paragraph*{}
Let $M$ be an unoriented  Riemannian manifold, and let $\widetilde{M} \to M$ be a two-fold covering by oriented  Riemannian manifold $\widetilde{M}$ with Spin-structure. There is an action of $\Z_2 \times \widetilde{M} \to \widetilde{M}$ such that $M \cong \widetilde{M}/ \Z_2$. Let $\widetilde{\SS}$ be a Spin-bundle and $\left(\Coo\left( \widetilde{M}\right) , L^2\left(\widetilde{M}, \widetilde{\SS}\right), \widetilde{\slashed D}  \right)$ the spectral triple. Suppose  $g\in \Z_2$ is the unique nontrivial element and there is an $\Z_2$-equivariant action $\Z_2 \times L^2\left(\widetilde{M}, \widetilde{\SS}\right)\to L^2\left(\widetilde{M}, \widetilde{\SS}\right)$, i.e.
\be\nonumber
\begin{split}
g \left( \widetilde{a}\widetilde{\xi}\right)= \left(g \widetilde{a}\right) \left(g \widetilde{\xi}\right); ~~ \forall \widetilde{a} \in C\left(\widetilde{M} \right), ~~ \forall \widetilde{\xi} \in L^2\left(\widetilde{M}, \widetilde{\SS}\right),\\
\left( g \widetilde{\xi}, g\widetilde{\eta}\right) = \left( \widetilde{\xi}, \widetilde{\eta}\right);~~\forall \widetilde{\xi}, \widetilde{\eta} \in L^2\left(\widetilde{M}, \widetilde{\SS}\right);\\ \text{ where } \left(\cdot, \cdot \right) \text{ is the scalar product on }L^2\left(\widetilde{M}, \widetilde{\SS}\right).
\end{split}
\ee
Suppose that $\widetilde{\slashed D}$ is $\Z_2$-invariant, i.e.
$$
g \left(\widetilde{\slashed D}\widetilde{\xi} \right) = \widetilde{\slashed D}\left(g \widetilde{\xi} \right); ~~  \forall \widetilde{\xi} \in \Dom \widetilde{\slashed D}. 
$$
Denote by
\be\nonumber
\begin{split}
L^2\left(\widetilde{M}, \widetilde{\SS}\right)^{\Z_2}= \left\{\widetilde{\xi} \in L^2\left(\widetilde{M}, \widetilde{\SS}\right)~|~g\widetilde{\xi} = \widetilde{\xi}\right\},\\
\slashed D = \widetilde{\slashed D}|_{L^2\left(\widetilde{M}, \widetilde{\SS}\right)^{\Z_2}}.
\end{split}
\ee
The  Riemannian manifold $M$ is unoriented , it is reasonable to say that $$\left(\Coo\left(M \right), L^2\left(\widetilde{M}, \widetilde{\SS}\right)^{\Z_2}, \slashed D  \right)$$ is an unoriented spectral triple.  Following definition is motivated by the above construction.
\begin{definition}\label{unoriented_defn}
	Denote by $g \in \Z_2$ the unique nontrivial element.
An {\it unoriented spectral triple} $(\A, \H, D)$ consists of:
\begin{enumerate}
	\item
	a pre-$C^*$-algebra $\A$ with an involution $a \mapsto a^*$, equipped
	with a faithful representation on:
	\item
	a \emph{Hilbert space} $\H$; and also
	\item
	a \emph{selfadjoint operator} $D$ on $\mathcal{H}$, with dense domain
	$\Dom D \subset \H$, such that $a(\Dom D) \subseteq \Dom D$ for all 
	$a \in \mathcal{A}$.
\item An unital oriented spectral triple $\left(\widetilde{\A}, \widetilde{\H}, \widetilde{D} \right)$ which satisfies to described in \cite{hajac:toknotes,varilly:noncom} axioms, such that following conditions hold:
\begin{enumerate}
	\item  There are  actions $\Z_2 \times \widetilde{\A} \to  \widetilde{\A} $,  $\Z_2 \times \widetilde{\H} \to  \widetilde{\H} $, such that
\be\label{main_def_1_eqn}
\begin{split}
	g \left( \widetilde{a}\widetilde{\xi}\right)= \left(g \widetilde{a}\right) \left(g \widetilde{\xi}\right); ~~ \forall \widetilde{a} \in \widetilde{\A}, ~~ \forall \widetilde{\xi} \in\widetilde{\H},\\
	\left( g \widetilde{\xi}, g\widetilde{\eta}\right) = \left( \widetilde{\xi}, \widetilde{\eta}\right);~~\forall \widetilde{\xi}, \widetilde{\eta} \in \widetilde{\H},,\\ \text{ where } \left(\cdot, \cdot \right) \text{ is the scalar product on }\widetilde{\H},
\end{split}
\ee

\be\label{main_def_2_eqn}
\begin{split}
g \left(\widetilde{ D}\widetilde{\xi} \right) = \widetilde{ D}\left(g \widetilde{\xi} \right); ~~  \forall \widetilde{\xi} \in \Dom \widetilde{ D}. 
\end{split}
\ee
\item There are isomorphisms
\be\label{main_def_3_eqn}
\begin{split}
\A \cong \widetilde{\A}^{\Z_2} \stackrel{\text{def}}{=}\left\{\widetilde{a} \in \widetilde{\A}~|~ g\widetilde{a} = \widetilde{a} \right\},\\
\H \cong \widetilde{\H}^{\Z_2}\stackrel{\text{def}}{=} \left\{\widetilde{\xi} \in \widetilde{\H}~|~ g\widetilde{\xi} = \widetilde{\xi} \right\}.
\end{split}
\ee
\item If $A$ (resp. $\widetilde{A}$) is a $C^*$-norm completion of $\A$ (resp. $\widetilde{\A}$) then the triple $\left(A, \widetilde{A}, \Z_2 \right)$ is an unital noncommutative finite-fold  covering and a following condition holds
\be\label{main_def_4_eqn}
\begin{split}
\A= A\bigcap \widetilde{\A}.
\end{split}
\ee
\item 
\be\label{main_def_5_eqn}
\begin{split}
D = \widetilde{ D}|_{\H}=\widetilde{ D}|_{\widetilde{\H}^{\Z_2}}.
\end{split}
\ee
\end{enumerate}

\end{enumerate}
\end{definition}

\section{Examples}
\subsection{Commutative unoriented spectral triples}\label{comm_sp_tr} 
\paragraph*{}
Let $M$ be an unoriented  Riemannian manifold, and let $\widetilde{M} \to M$ be a two listed covering by oriented  Riemannian manifold $\widetilde{M}$ with Spin-structure given by Spin-bundle $\SS$. From the construction of the Section \ref{main_defn_sec} it follows that there is a commutative unoriented spectral triple 
\be\label{comm_equ}
\left(\Coo\left(M \right), L^2\left(\widetilde{M}, \widetilde{\SS}\right)^{\Z_2}, \slashed D  \right).
\ee
\subsection{Quantum $SO\left(3 \right)$}
\paragraph*{} 
Denote by $g \in \Z_2$ the unique nontrivial element. There is a surjective group homomorphism
$$
\Phi : SU\left(2 \right) \to SO\left(3 \right) ,~ \ker \Phi = \Z_2 = \{\pm 1\}
$$
and the natural action of $\Z_2$ on $SU(2)$ such that
\be\label{su_q_2_z2_comm_eqn}
\begin{split}
SO\left(3 \right) \cong SU\left(2\right)/\Z_2,\\
g\begin{pmatrix} \al & -\overline{\bt}\\
\bt & \overline{\al}\end{pmatrix}= \begin{pmatrix} -\al & \overline{\bt}\\
-\bt & -\overline{\al}\end{pmatrix};~ \forall \begin{pmatrix} \al & -\overline{\bt}\\
\bt & \overline{\al}\end{pmatrix}\in SU\left(2\right).
\end{split}
\ee	
This action induces an action of $\Z_2$  on a $C^*$-algebra $C\left(SU\left(2\right) \right)$ given by
\be\nonumber
\begin{split}
g\al = -\al, ~~ g\bt = - \bt
\end{split}
\ee
where $\al, \bt$ are regarded as functions $\SU\left(2 \right) \to \C$.
Indeed 	$ SU\left(2 \right)$ is an oriented manifold, $SO\left(3 \right)$ is an unoriented one, and $ SU\left(2 \right) \to SO\left(3 \right)$ is a two-fold covering.
There is a quantum generalization of $SU\left(2 \right)$ and  we will introduce a quantum analog of $SO\left(3 \right)$.
Let $q$ be a real number such that $0<q<1$. 
A quantum group $C\left( \SU_q(2)\right) $ is the universal $C^*$-algebra algebra generated by two elements $\al$ and $\beta$ satisfying the following relations:
\begin{equation}\label{su_q_2_rel_eqn}
\begin{split}
\al^*\al + \beta^*\beta = 1, ~~ \al\al^* + q^2\beta\beta^* =1,
\\
\al\bt - q \bt\al = 0, ~~\al\bt^*-q\bt^*\al = 0,
\\
\bt^*\bt = \bt\bt^*.
\end{split}
\end{equation}
From  $C\left( SU_1\left(2 \right)\right) \approx C\left(SU\left(2 \right)  \right)$ it follows that  $C\left( \SU_q(2)\right) $ can be regarded as a noncommutative deformation of $SU(2)$. The dense pre-$C^*$-algebra $\Coo\left( SU_q(2)\right) \subset C\left( SU_q(2)\right)$ is defined in  \cite{chakraborty_pal:quantum_su_2}. 
Let $Q, S \in B\left( \ell_2\left(\N^0 \right)\right) $ be given by 
\begin{equation*}
\begin{split}
Qe_k= q^ke_k, \\
Se_k = \left\{
\begin{array}{c l}
e_{k-1} & k > 0 \\
0 & k = 0
\end{array}\right..
\end{split}
\end{equation*}
and let $R \in B\left( \ell_2\left(\Z \right)\right) $ be given by $e_k \mapsto e_{k+1}$.
There is a faithful representation \cite{woronowicz:su2}  $C\left(\SU_q\left( 2\right) \right) \to B\left(\ell_2\left(\N^0 \right) \otimes \ell_2\left(\Z \right) \right)  $ given by

\begin{equation}\label{su_q_2_repr_eqn}
\begin{split}
\al \mapsto S\sqrt{1 - Q^2} \otimes 1, \\
\bt \mapsto Q \otimes R.
\end{split}
\end{equation}
We will denote by $\A_f$ the dense $*$-subalgebra of $\Coo\left( SU_q(2)\right)$ generated
by $\alpha$ and $\beta$. 
There is a faithful state $h:C\left( SU_q(2)\right) \to \C$ given by
\be\label{su_q_2_haar_eqn}
h\left(a \right)  = \sum_{n = 0}^\infty q^{2n}\left(e_n \otimes e_0, a  e_n \otimes e_0\right) 
\ee
where $a \in C\left( SU_q(2)\right)$ and $e_0 \otimes e_n \in \ell_2\left(\N^0 \right) \otimes \ell_2\left(\Z \right)$ (cf. \cite{woronowicz:su2}).
\begin{defn}
	The state $h$ is said to be the \textit{Haar measure}.
\end{defn}

Denote by $L^2\left( C\left( SU_q\left(2\right)\right), h\right) $ the GNS space associated with the state $h$. 
The representation theory of $SU_q(2)$ is strikingly similar to its
classical counterpart.  In particular, for each $l\in\{0,\frac{1}{2},
1,\ldots\}$, there is a unique irreducible unitary representation
$t^{(l)}$ of dimension $2n+1$.  Denote by $t^{(l)}_{jk}$ the
$jk$\raisebox{.4ex}{th} entry of $t^{(l)}$. These are all elements of
$\A_f$ and they form an orthogonal basis for $L^2\left( C\left( SU_q\left(2\right)\right), h\right)$. Denote by
$e^{(l)}_{jk}$ the normalized $t^{(l)}_{jk}$'s, so that
$\{e^{(l)}_{jk}: n=0,\frac{1}{2},1,\ldots, i,j=-n,-n+1,\ldots, n\}$ is
an orthonormal basis. The definition of equivariant operators (with respect to action of quantum groups) is described in \cite{chakraborty_pal:inv_hom}. It is proven in \cite{chakraborty_pal:quantum_su_2} that any unbounded equivariant operator $\widetilde{D}$ satisfies to the following condition
\be \label{su_q_2_dirac_eqn}
\widetilde{D}: e^{(l)}_{jk}\mapsto d(l,j)e^{(l)}_{jk},
\ee
Moreover if
\be \label{su_q_2_genericd_eqn}
d(l,j)=\begin{cases}2l+1 &  l \neq j,\cr
	-(2l+1) & l=j,\end{cases}
\ee

then there is a 3-summable spectral triple 
\be \label{su_q_2_spt_eqn}
\left(\Coo\left( SU_q\left(2\right)\right), L^2\left( C\left( SU_q\left(2\right)\right), h\right), \widetilde{D} \right) 
\ee
described in \cite{chakraborty_pal:quantum_su_2}. 
According to \cite{kl-sch} (equations (4.42)-(4.44) ) following condition holds
\begin{equation}\label{su_q_2_tij_eqn}
\begin{split}
t^{\left(l \right) }_{jk}= N^l_{jk}\al^{-j-k}\bt^{k-j}p_{l + k}\left( \bt\bt^*; q^{-2\left(k-j \right) }q^{2\left(j + k \right)}~|~q^2 \right); ~~ j + k \le 0 ~\&~ k \ge j,\\ t^{\left(l \right) }_{jk}= N^l_{-j,-k}p_{l - k}\left( \bt\bt^*; q^{-2\left(k-j \right) }q^{2\left(j + k \right)}~|~q^2 \right)\bt^{k-j}\al^{*j+k}; ~~ j + k \ge 0 ~\&~ k \ge j,\\ 
t^{\left(l \right) }_{jk}= N^l_{-k,-j}p_{l - j}\left( \bt\bt^*; q^{-2\left(k-j \right) }q^{2\left(j + k \right)}~|~q^2 \right)\bt^{*j-k}\al^{*j + k}; ~~ j + k \ge 0 ~\&~ j \ge k\\ 
\end{split}
\end{equation}
where $N^l_{jk} \in \R$ for any $l, j, k$ and $p_{l - k}\left( x; q^{-2\left(k-j \right) }q^{2\left(j + k \right)}~|~q^2 \right)$ is little Jacobi polynomial (cf. \cite{kl-sch}).  There is a noncommutative analog of the action \eqref{su_q_2_z2_comm_eqn} given by
\begin{equation}\label{su_q_2_z2_ncomm_eqn}
\begin{split}
\Z_2 \times C\left( SU_q\left(2\right)\right)\to C\left( SU_q\left(2\right)\right),\\
g\al = -\al,~g \bt = -\bt.
\end{split}
\end{equation}
\begin{definition}
	Denote by
	\begin{equation}\label{su_q_2_so3_eqn}
C\left( 	SO_q\left(3 \right) \right) \stackrel{\mathrm{def}}{=} C\left( SU_q\left(2\right)\right) ^{\Z_2}\cong \left\{\widetilde{a}\in C\left( SU_q\left(2\right)\right) ,~ g\widetilde{a} = \widetilde{a} \right\}.
	\end{equation}
	The $C^*$-algebra $C\left( 	SO_q\left(3 \right)\right) $ is said to be the \textit{quantum 	$ SO\left(3 \right)$}.
\end{definition}

\begin{thm}\label{su_q_2_bas_thm}\cite{woronowicz:su2}
	Let  $q \neq 0$. The set of elements of the form
	\be\label{su_q_2_fin_eqn}
	\al^k\bt^n\bt^{*m}~~ \text{ and }~~ \al^{*k'}\bt^n\bt^{*m}
	\ee 
	where $k, m, n = 0, \dots;~k'=1, 2,\dots$ forms a basis in $\A_f$: any element
	of $\A_f$ can be written in the unique way as a finite linear combination of
	elements of \eqref{su_q_2_fin_eqn}.
\end{thm}
\begin{lem}
	$C\left( SU_q\left(2\right)\right)$ is a finitely generated projective $C\left( 	SO_q\left(3 \right) \right)$ module.
\end{lem}
\begin{proof}
	If $A_f^{\Z_2}= A_f\bigcap C\left( 	SO_q\left(3 \right) \right) $ then from \eqref{su_q_2_z2_ncomm_eqn} and the Theorem \ref{su_q_2_bas_thm} it turns out that given by \eqref{su_q_2_fin_eqn} elements
\be\nonumber
\al^k\bt^n\bt^{*m}~~ \text{ and }~~ \al^{*k'}\bt^n\bt^{*m}
\ee 
with even $k+m+n$ or $k'+m+n$ is the basis of $A_f^{\Z_2}$.	If 
$$
\widetilde{a} = \al^k\bt^n\bt^{*m} \notin A_f^{\Z_2}
$$
then $k+m+n$ is odd. If $m > 0$ then
$$
\widetilde{a} = \al^k\bt^n\bt^{*m-1} \bt^* = a \bt^* \text{ where } a \in A_f^{\Z_2}.
$$
If $m = 0$ and $n > 0$ then
$$
\widetilde{a} = \al^k\bt^{n-1} \bt = a \bt \text{ where } a \in A_f^{\Z_2}
$$
If $m = 0$ and $n= 0$ then $k > 0$ and
$$
\widetilde{a} = \al^{k-1}\al = a \al \text{ where } a \in A_f^{\Z_2}.
$$
From 
$$
\widetilde{a} = \al^{*k'}\bt^n\bt^{*m} \notin A_f^{\Z_2}
$$
it follows that  $k'+m+n$ is odd. Similarly to the above proof one has
$$
\widetilde{a} = a \al\text{ or } \widetilde{a} = a \al^* \text{ or } \widetilde{a} = a \bt \text{ or } \widetilde{a} = a \bt^*   \text{ where } a \in A_f^{\Z_2}.
$$
From the above equations it turns out that $A_f$ is a left $A_f^{\Z_2}$-module generated by $\al, \al^*, \bt, \bt^*$. Algebra $A_f^{\Z_2}$ (resp. $A_f$) is dense in $ C\left( 	SO_q\left(3 \right) \right) $ (resp.  $C\left( 	SU_q\left(2 \right) \right)$ ) it follows that  $C\left( 	SU_q\left(2 \right) \right)$ is a left $C\left( 	SO_q\left(3 \right) \right)$-module generated by $\al, \al^*, \bt, \bt^*$. From the Remark \ref{kasparov_stab_thm} it turns out that $C\left( SU_q\left(2\right)\right)$ is a finitely generated projective $C\left( 	SO_q\left(3 \right) \right)$ module.

\end{proof}
\begin{corollary}
The	triple $\left(C\left( 	SO_q\left(3 \right) \right), C\left( 	SU_q\left(2 \right) \right), \Z_2 \right)$ is an unital noncommutative finite-fold  covering.
\end{corollary}
An action of $\Z_2$ on $L^2\left( C\left( SU_q\left(2\right)\right), h\right)$ is naturally induced by the action  $\Z_2$ on $\Coo\left( SU_q\left(2\right)\right) $. From the above construction it follows that the unital orientable spectral triple
\be\nonumber
\left(\Coo\left( SU_q\left(2\right)\right), L^2\left( C\left( SU_q\left(2\right)\right), h\right), \widetilde{D} \right) 
\ee
can be regarded as the triple given by condition 4 of the Definition  \ref{unoriented_defn}. Also one sees that all conditions of the the Definition  \ref{unoriented_defn} hold, so one has an unoriented spectral triple
$$
\left(\Coo\left( SO_q\left(3\right)\right), L^2\left( C\left( SU_q\left(2\right)\right), h\right)^{\Z_2},D \right) 
$$
where $\Coo\left( SO_q\left(3\right)\right)= C\left( SO_q\left(3\right)\right)\bigcap \Coo\left( SU_q\left(2\right)\right)$ and $D = \widetilde{ D}|_{L^2\left( C\left( SU_q\left(2\right)\right), h\right)^{\Z_2}}$.
\subsection{Isopectral deformations}

\subsubsection{Oriented Twisted Spectral Triples}\label{iso_ori}
\paragraph*{}A very general construction of isospectral
deformations
of noncommutative geometries is described in \cite{connes_landi:isospectral}. The construction
implies in particular that any
compact Spin-manifold $M$ whose isometry group has rank
$\geq 2$ admits a
natural one-parameter isospectral deformation to noncommutative geometries
$M_\theta$.
We let $(\Coo\left(M \right)  , \H = L^2\left(M,S \right)  , \slashed D)$ be the canonical spectral triple associated with a
compact spin-manifold $M$. We recall that $\mathcal{A} = C^\infty(M)$ is
the algebra of smooth
functions on $M$, $S$ is the spinor bundle and $\slashed D$
is the Dirac operator.
Let us assume that the group $\mathrm{Isom}(M)$ of isometries of $M$ has rank
$r\geq2$.
Then, we have an inclusion
\begin{equation}\label{isos_t_act_eqn}
\mathbb{T}^2 \subset \mathrm{Isom}(M) \, ,
\end{equation}
with $\mathbb{T}^2 = \mathbb{R}^2 / 2 \pi \mathbb{Z}^2$ the usual torus, and we let $U(s) , s \in
\mathbb{T}^2$, be
the corresponding unitary operators in $\H = L^2(M,S)$ so that by construction
\begin{equation*}
U(s) \, \slashed D = \slashed D \, U(s).
\end{equation*}
Also,
\begin{equation}\label{isospectral_sym_eqn}
U(s) \, a \, U(s)^{-1} = \alpha_s(a) \, , \, \, \, \forall \, a \in \mathcal{A} \, ,
\end{equation}
where $\alpha_s \in \mathrm{Aut}(\mathcal{A})$ is the action by isometries on the
algebra of functions on
$M$.

\noindent
We let $p = (p_1, p_2)$ be the generator of the two-parameters group $U(s)$
so that
\begin{equation*}
U(s) = \exp(i(s_1 p_1 + s_2 p_2)) \, .
\end{equation*}
The operators $p_1$ and $p_2$ commute with $D$.
Both $p_1$ and $p_2$
have integral spectrum,
\begin{equation*}
\mathrm{Spec}(p_j) \subset \mathbb{Z} \, , \, \, j = 1, 2 \, .
\end{equation*}

\noindent
One defines a bigrading of the algebra of bounded operators in $\H$ with the
operator $T$ declared to be of bidegree
$(n_1,n_2)$ when,
\begin{equation*}
\alpha_s(T) = \exp(i(s_1 n_1 + s_2 n_2)) \, T \, , \, \, \, \forall \, s \in
\mathbb{T}^2 \, ,
\end{equation*}
where $\alpha_s(T) = U(s) \, T \, U(s)^{-1}$ as in \eqref{isospectral_sym_eqn}.
\paragraph{}
Any operator $T$ of class $C^\infty$ relative to $\alpha_s$ (i. e. such that
the map $s \rightarrow \alpha_s(T) $ is of class $C^\infty$ for the
norm topology) can be uniquely
written as a doubly infinite
norm convergent sum of homogeneous elements,
\begin{equation*}
T = \sum_{n_1,n_2} \, \widehat{T}_{n_1,n_2} \, ,
\end{equation*}
with $\widehat{T}_{n_1,n_2}$ of bidegree $(n_1,n_2)$ and where the sequence
of norms $||
\widehat{T}_{n_1,n_2} ||$ is of
rapid decay in $(n_1,n_2)$.
Let $\lambda = \exp(2 \pi i \theta)$. For any operator $T$ in $\H$ of
class $C^\infty$ we define
its left twist $l(T)$ by
\begin{equation}\label{l_defn}
l(T) = \sum_{n_1,n_2} \, \widehat{T}_{n_1,n_2} \, \lambda^{n_2 p_1} \, ,
\end{equation}
and its right twist $r(T)$ by
\begin{equation*}
r(T) = \sum_{n_1,n_2} \, \widehat{T}_{n_1,n_2} \, \lambda^{n_1 p_2} \, ,
\end{equation*}
Since $|\lambda | = 1$ and $p_1$, $p_2$ are self-adjoint, both series
converge in norm. Denote by $\Coo\left(M \right)_{n_1, n_2} \subset \Coo\left(M \right) $ the $\C$-linear subspace of elements of bidegree $\left( n_1, n_2\right) $. \\
One has,
\begin{lem}\label{conn_landi_iso_lem}\cite{connes_landi:isospectral}
	\begin{itemize}
		\item[{\rm a)}] Let $x$ be a homogeneous operator of bidegree $(n_1,n_2)$
		and $y$ be
		a homogeneous operator of  bidegree $(n'_1,n'_2)$. Then,
		\begin{equation}
		l(x) \, r(y) \, - \,  r(y) \, l(x) = (x \, y \, - y \, x) \,
		\lambda^{n'_1 n_2} \lambda^{n_2 p_1 + n'_1 p_2}
		\end{equation}
		In particular, $[l(x), r(y)] = 0$ if $[x, y] = 0$.
		\item[{\rm b)}] Let $x$ and $y$ be homogeneous operators as before and
		define
		\begin{equation}
		x * y = \lambda^{n'_1 n_2} \, x y \, ; \label{star}
		\end{equation}
		then $l(x) l(y) = l(x * y)$.
	\end{itemize}
\end{lem}

\noindent
The product $*$ defined in (\ref{star}) extends by linearity
to an associative product on the linear space of smooth operators and could
be called a $*$-product.
One could also define a deformed `right product'. If $x$ is homogeneous of
bidegree
$(n_1,n_2)$ and $y$ is homogeneous of bidegree $(n'_1,n'_2)$ the product is
defined by
\begin{equation*}
x *_{r} y = \lambda^{n_1 n'_2} \, x y \, .
\end{equation*}
Then, along the lines of the previous lemma one shows that $r(x) r(y) = r(x
*_{r} y)$.

We can now define a new spectral triple where both $\H$ and the operator
$D$ are unchanged while the
algebra $\Coo\left(M \right)$  is modified to $l(\Coo\left(M \right))$ . By
Lemma~{\ref{conn_landi_iso_lem}}~b) one checks that  $l\left( \Coo\left(M \right)\right) $ is still an algebra. Since $\slashed D$ is of bidegree $(0,0)$ one has,
\begin{equation*}
[\slashed D, \, l(a) ] = l([\slashed D, \, a]) \label{bound}
\end{equation*}
which is enough to check that $[\slashed D, x]$ is bounded for any $x \in l(\mathcal{A})$. There is an oriented twisted spectral triple 
\be\label{twisetd_eqn}
\left(l\left( \Coo\left(M \right)\right) , \H, \slashed D\right).
\ee
\subsubsection{Unoriented Twisted Spectral Triples}
\paragraph*{}
Suppose that $M$ is unoreintable manifold which satisfies to \eqref{isos_t_act_eqn}, i.e.
\begin{equation*}
\mathbb{T}^2 \subset \mathrm{Isom}(M) \, ,
\end{equation*}
Suppose that the natural 2-fold covering $\widetilde{M}\to M$ is such that $\widetilde{M}$ is a Spin-manifold so there is an oriented spectral triple
$\left( \Coo\left(\widetilde{M}\right) , L^2\left(\widetilde{M},\widetilde{S} \right), \widetilde{ \slashed D}  \right)$.
From \ref{comm_sp_tr} it turns out that there is an unoriented spectral triple given by \eqref{comm_equ}, i.e.

\be\nonumber
\left(\Coo\left(M \right), L^2\left(\widetilde{M}, \widetilde{\SS}\right)^{\Z_2}, \slashed D  \right).
\ee
Otherwise from \ref{iso_ori} there is an oriented twisted spectral triple
\be\nonumber
\left(l \Coo\left(\widetilde{M}\right) , L^2\left(\widetilde{M},\widetilde{S} \right), \widetilde{ \slashed D}  \right).
\ee.

Action of $G\left(\widetilde{M}~|~M \right) \cong \Z_2$ on $\widetilde{M}$ induces an action of $\Z_2$ on both $\Coo\left(\widetilde{M}\right) $ and $l\Coo\left(\widetilde{M}\right) $ such that
\be\nonumber
\begin{split}
	\Coo\left(M \right)= \Coo\left(\widetilde{M}\right)^{\Z_2},\\
	l\Coo\left(M \right)= l\Coo\left(\widetilde{M}\right)^{\Z_2},\\
\end{split}
\ee
From the above construction we have an unoriented twisted spectral triple
$$
\left(l\Coo\left(M \right), L^2\left(\widetilde{M}, \widetilde{\SS}\right)^{\Z_2}, \slashed D  \right).
$$

 which satisfies to the Definition \eqref{unoriented_defn}.

\section*{Acknowledgment}

\paragraph*{}
I am very grateful to Prof.  Arup Kumar Pal and Stanis\l{}aw Lech Woronowicz
for explanation of the properties of the quantum $SU\left( 2\right)$. I would like to acknowledge my son, Ph. D. Nikolay Ivankov for the discussion about this article.

\end{document}